%% file: cubical-homology-arXiv.tex
\documentclass[oneside,article]{memoir}

\usepackage{amsmath}         
\usepackage{amsthm,thmtools} 
\usepackage{enumitem}        

\usepackage{tikz}
  \usetikzlibrary{matrix,arrows}
  
\usepackage{tikz-cd}

\usepackage{float} 

\usepackage[nobysame,alphabetic]{amsrefs} 
\usepackage[hidelinks]{hyperref}                   
\usepackage[capitalize,nosort]{cleveref}           

\usepackage{amssymb}
\usepackage[mathscr]{euscript}
\usepackage{relsize}
\usepackage{stmaryrd}
\usepackage{stackengine}

\usepackage{indentfirst} 
\usepackage{multicol}

\usepackage[inline,nomargin]{fixme} 
  \fxsetup{targetlayout=color}

\usepackage{float} 

\usepackage{lmodern}
\usepackage{libertine}
\usepackage[libertine]{newtxmath}

\isopage[12]

\setlrmargins{*}{*}{1}
\checkandfixthelayout

\counterwithout{section}{chapter}
\usepackage{appendix}

\input{defs}

\newcommand{\cSet}{\mathsf{cSet}} 
\newcommand{\Grp}{\mathsf{Grp}}
\newcommand{\Set}{\mathsf{Set}}
\newcommand{\Ab}{\mathsf{Ab}}
\newcommand{\Ch}{\mathsf{Ch}}
\newcommand{\Kan}{\mathsf{Kan}}

\newcommand{\fcat}[2]{{#2}^{#1}} 
\renewcommand{\Top}{\ncat{Top}}

\newcommand{\boxcat}{\mathord{\square}} 
\newcommand{\face}[2]{\partial^{#1}_{#2}} 
\newcommand{\degen}[2]{\sigma^{#1}_{#2}} 
\newcommand{\conn}[2]{\gamma^{#1}_{#2}} 
\newcommand{\ndegen}[1][n]{\sigma^{\circ #1}} 
\newcommand{\cube}[1]{\mathord{\square^{#1}}} 
\newcommand{\obox}[2]{\mathord{\sqcap^{#1}_{#2}}} 
\newcommand{\dfobox}[1][n]{\mathord{\sqcap^{#1}_{i,\varepsilon}}} 
\newcommand{\reali}[2][]{\lvert #2 \rvert_{#1}} 
\newcommand{\lhom}{\operatorname{hom}_{L}} 
\newcommand{\rhom}{\operatorname{hom}_{R}} 

\newcommand{\Sing}{\operatorname{Sing}} 

\newcommand{\loopsp}[1][]{\Omega^{#1}} 
\newcommand{\cpsq}[1]{\langle #1 \rangle} 

\newcommand{\Z}{\mathbb{Z}} 
\newcommand{\cmap}[1][]{\partial_{#1}} 
\newcommand{\im}{\operatorname{im}} 
\newcommand{\cAb}{\mathsf{cAb}} 
\newcommand{\D}[1][]{D_{#1}} 
\newcommand{\C}{\delta} 
\newcommand{\pH}{\tilde{H}} 

\newcommand{\restr}[2]{{#1}|_{#2}} 
\newcommand{\noproof}{\hfill\qedsymbol}

\author{Daniel Carranza \and Krzysztof Kapulkin \and Andrew Tonks}
\title{The Hurewicz theorem for cubical homology}
\date{\today}

\begin{document}

  \maketitle

  \begin{abstract}
     We give an elementary proof of the Hurewicz theorem relating homotopy and homology groups of a cubical Kan complex.
     Our approach is based on the notion of a loop space of a cubical set, developed in a companion paper ``Homotopy groups of cubical sets'' by the first two authors.
     \MSC{55N10 (primary), 55P35, 55U35 (secondary)}
  \end{abstract}

\input{0-introduction.tex}

\input{1-csets.tex}

\input{2-homotopy.tex}

\input{3-homology.tex}

\input{4-hurewicz.tex}

\bibliographystyle{amsalphaurlmod}
\bibliography{all-refs.bib}

\end{document}

%% file: defs.tex
  \newcommand{\op}{{\mathord\mathsf{op}}}

  \newcommand{\slice}{\mathbin\downarrow}

  \newcommand{\bd}{\partial}

  \newcommand{\gprod}{\otimes}

  \newcommand{\id}[1][]{\operatorname{id}_{#1}}
  
  \makeatletter
  \def\horn#1{\expandafter\horn@i#1,,\@nil}
  \def\horn@i#1,#2,#3\@nil{\Lambda^{#2}[#1]}
  \makeatother

  \newcommand{\uvar}{\mathord{\relbar}}

  \renewcommand{\tilde}{\widetilde}

  \newcommand{\N}{{\mathord\mathbb{N}}}

  \newcommand{\ncat}[1]{\mathsf{#1}}

  \newcommand{\sSet}{\ncat{sSet}}

  \newcommand{\from}{\colon}

  \newcommand{\ito}{\hookrightarrow}

  \declaretheorem[style=definition,within=section]{definition}
  \declaretheorem[style=definition,numberlike=definition]{example}
  \declaretheorem[style=definition,numberlike=definition]{remark}
  \declaretheorem[style=definition,numberlike=definition]{notation}

  \declaretheorem[style=plain,numberlike=definition]{corollary}
  \declaretheorem[style=plain,numberlike=definition]{lemma}
  \declaretheorem[style=plain,numberlike=definition]{proposition}
  \declaretheorem[style=plain,numberlike=definition]{theorem}

  \declaretheorem[style=plain,numbered=no,name=Theorem]{theorem*}

  \Crefname{corollary}{Corollary}{Corollaries}
  \Crefname{definition}{Definition}{Definitions}
  \Crefname{lemma}{Lemma}{Lemmas}
  \Crefname{proposition}{Proposition}{Propositions}
  \Crefname{remark}{Remark}{Remarks}
  \Crefname{theorem}{Theorem}{Theorems}

  \newlist{axioms}{enumerate}{1}
  \Crefname{axiomsi}{}{}

  \newenvironment{tikzeq*}
  {
    \begingroup
    \begin{equation*}
    \begin{tikzpicture}[baseline=(current bounding box.center)]
  }
  {
    \end{tikzpicture}
    \end{equation*}
    \endgroup
    \ignorespacesafterend
  }

  \tikzset
  {
    diagram/.style=
    {
      matrix of math nodes,
      column sep={4.3em,between origins},
      row sep={4em,between origins},
      text height=1.5ex,
      text depth=.25ex
    },
    over/.style={preaction={draw=white,-,line width=6pt}},
    every to/.style={font=\footnotesize},
    inj/.style={right hook->},
    surj/.style={-{Latex[open]}},
    cof/.style={>->},
    fib/.style={->>},
  }

  \DeclareFontFamily{U}{mathx}{\hyphenchar\font45}

  \DeclareFontShape{U}{mathx}{m}{n}{
    <5> <6> <7> <8> <9> <10>
    <10.95> <12> <14.4> <17.28> <20.74> <24.88>
    mathx10}{}

  \DeclareSymbolFont{mathx}{U}{mathx}{m}{n}

  \DeclareFontFamily{U}{mathb}{\hyphenchar\font45}

  \DeclareFontShape{U}{mathb}{m}{n}{
    <5> <6> <7> <8> <9> <10>
    <10.95> <12> <14.4> <17.28> <20.74> <24.88>
    mathb10}{}

  \DeclareSymbolFont{mathb}{U}{mathb}{m}{n}

  \DeclareMathAccent{\wbar}{0}{mathx}{"73}

  \DeclareMathSymbol{\Rsh}{\mathrel}{mathb}{"E9}

  \DeclareFontFamily{U}{MnSymbolA}{}

  \DeclareFontShape{U}{MnSymbolA}{m}{n}{
    <-6> MnSymbolA5
    <6-7> MnSymbolA6
    <7-8> MnSymbolA7
    <8-9> MnSymbolA8
    <9-10> MnSymbolA9
    <10-12> MnSymbolA10
    <12-> MnSymbolA12}{}

  \DeclareSymbolFont{MnSyA}{U}{MnSymbolA}{m}{n}

  \DeclareMathSymbol{\twoheaddownarrow}{\mathrel}{MnSyA}{27}

  \newcommand{\MSC}[1]{%
    \let\thempfn\relax
    \footnotetext[0]{2020 Mathematics Subject Classification: #1.}
  }

%% file: 0-introduction.tex
\section*{Introduction}

The Hurewicz theorem is among the most basic, yet powerful tools of homotopy theory.
Given an $(n-1)$-connected pointed space $(X, x_0)$, i.e., a space such that $\pi_i(X, x_0)=\{*\}$ for $i < n$, the Hurewicz theorem gives an isomorphism $\pi_n(X, x_0) \cong \tilde{H}_n(X, x_0)$ between its first non-trivial homotopy group and the corresponding (reduced) homology group.

In this paper, we give a proof of this theorem for cubical sets, a combinatorial model for the category of topological spaces.
Cubical sets in many ways resemble simplicial sets, but while the simplices are built out of the interval by taking cones, the cubes are built out of the interval by taking products.

The notions of homotopy and homology groups of simplicial sets are, by now, well-established in the literature \cites{goerss-jardine,may:simplicial-objects}.
In particular, both of these references include a proof of the Hurewicz theorem for simplicial sets (cf.~\cite[Thm.~III.3.7]{goerss-jardine} and \cite[Thm.~13.6]{may:simplicial-objects})
Similarly, the notion of cubical homology has a long history and can be found in a number of textbooks, including \cite{massey:basic-course}, and in the context of cubical sets, e.g., in \cite{barcelo-greene-jarrah-welker}.
More recently, homotopy groups of cubical sets were introduced in \cite{carranza-kapulkin:homotopy-cubical}, although the corresponding theory was perhaps already partially known to experts.
These are the necessary ingredients allowing us to state and prove the Hurewicz theorem for cubical sets.

Here, we work with cubical sets with both (positive and negative) connections, as used in \cite{tonks:cubical-kan,doherty-kapulkin-lindsey-sattler,barcelo-greene-jarrah-welker,carranza-kapulkin:homotopy-cubical}.
Connections, introduced by Brown and Higgins \cite{brown-higgins:algebra-of-cubes}, are additional degeneracy maps making the theory better-behaved (cf.~\cite{maltsiniotis:cubes-with-connections-strict-test-category}).
However, although our cube category has both connections, our proofs apply verbatim to the case of having only the negative connection, and can easily be adapted to the case of having only the positive connection.

We build on two recent papers: \cite{barcelo-greene-jarrah-welker} and \cite{carranza-kapulkin:homotopy-cubical}.
The first of these establishes key results on the homology of cubical sets with connections, relating multiple chain complexes obtained from a single cubical set by quotienting by neither, one, or both connections.
The second develops loop spaces of cubical Kan complexes, which are more convenient to work with, and allow us to give a more elementary and straightforward proof of the Hurewicz theorem than their simplicial counterparts.

This paper is organized as follows.
In \cref{sec:cset}, we recall the requisite background on cubical sets and (cubical) Kan complexes.
\cref{sec:homotopy} reviews the notions of loop spaces and homotopy groups of cubical sets. 
After that, we review the construction of cubical homology in \cref{sec:homology}.
Finally, in \cref{sec:hurewicz}, we put it all together by defining the Hurewicz map and proving the Hurewicz theorem, that for any $(n-1)$-connected pointed Kan complex $(X, x_0)$, the Hurewicz map $\pi_n(X, x_0) \to \tilde{H}_n(X, x_0)$ is an isomorphism.

%% file: 1-csets.tex
\section{Cubical sets and Kan complexes} \label{sec:cset}

In this section, we introduce cubical sets and (cubical) Kan complexes, alongside key constructions, such as the geometric product, thereon.

We begin by defining the box category $\Box$.
The objects of $\Box$ are posets of the form $[1]^n = \{ 0 \leq 1 \}^n$ and the maps are generated (inside the category of posets) under composition by the following four special classes:
\begin{itemize}
  \item \emph{faces} $\partial^n_{i,\varepsilon} \colon [1]^{n-1} \to [1]^n$ for $i = 1, \ldots , n$ and $\varepsilon = 0, 1$ given by:
  \[ \partial^n_{i,\varepsilon} (x_1, x_2, \ldots, x_{n-1}) = (x_1, x_2, \ldots, x_{i-1}, \varepsilon, x_i, \ldots, x_{n-1})\text{;}  \]
  \item \emph{degeneracies} $\sigma^n_i \colon [1]^n \to [1]^{n-1}$ for $i = 1, 2, \ldots, n$ given by:
  \[ \sigma^n_i ( x_1, x_2, \ldots, x_n) = (x_1, x_2, \ldots, x_{i-1}, x_{i+1}, \ldots, x_n)\text{;}  \]
  \item \emph{negative connections} $\gamma^n_{i,0} \colon [1]^n \to [1]^{n-1}$ for $i = 1, 2, \ldots, n-1$ given by:
  \[ \gamma^n_{i,0} (x_1, x_2, \ldots, x_n) = (x_1, x_2, \ldots, x_{i-1}, \max\{ x_i , x_{i+1}\}, x_{i+2}, \ldots, x_n) \text{.} \]
  \item \emph{positive connections} $\gamma^n_{i,1} \colon [1]^n \to [1]^{n-1}$ for $i = 1, 2, \ldots, n-1$ given by:
  \[ \gamma^n_{i,1} (x_1, x_2, \ldots, x_n) = (x_1, x_2, \ldots, x_{i-1}, \min\{ x_i , x_{i+1}\}, x_{i+2}, \ldots, x_n) \text{.} \]
\end{itemize}

These maps obey the following \emph{cubical identities}:

\[ \begin{array}{l l}
    \partial_{j, \varepsilon'} \partial_{i, \varepsilon} = \partial_{i+1, \varepsilon} \partial_{j, \varepsilon'} \quad \text{for } j \leq i; & 
    \sigma_j \partial_{i, \varepsilon} = \begin{cases}
        \partial_{i-1, \varepsilon} \sigma_j & \text{for } j < i; \\
        \id                                                       & \text{for } j = i; \\
        \partial_{i, \varepsilon} \sigma_{j-1} & \text{for } j > i;
    \end{cases} \\
    \sigma_i \sigma_j = \sigma_j \sigma_{i+1} \quad \text{for } j \leq i; &
    \gamma_{j,\varepsilon'} \gamma_{i,\varepsilon} = \begin{cases}
    \gamma_{i,\varepsilon} \gamma_{j+1,\varepsilon'} & \text{for } j > i; \\
    \gamma_{i,\varepsilon}\gamma_{i+1,\varepsilon} & \text{for } j = i, \varepsilon' = \varepsilon;
    \end{cases} \\
    \gamma_{j,\varepsilon'} \partial_{i, \varepsilon} = \begin{cases} 
        \partial_{i-1, \varepsilon} \gamma_{j,\varepsilon'}   & \text{for } j < i-1 \text{;} \\
        \id                                                         & \text{for } j = i-1, \, i, \, \varepsilon = \varepsilon' \text{;} \\
        \partial_{i, \varepsilon} \sigma_i         & \text{for } j = i-1, \, i, \, \varepsilon = 1-\varepsilon' \text{;} \\
        \partial_{i, \varepsilon} \gamma_{j-1,\varepsilon'} & \text{for } j > i;
    \end{cases} &
    \sigma_j \gamma_{i,\varepsilon} = \begin{cases}
        \gamma_{i-1,\varepsilon} \sigma_j  & \text{for } j < i \text{;} \\
        \sigma_i \sigma_i           & \text{for } j = i \text{;} \\
        \gamma_{i,\varepsilon} \sigma_{j+1} & \text{for } j > i \text{.} 
    \end{cases}
\end{array} \]
Every morphism in $\Box$ can be written in a \emph{standard form} as follows. 
\begin{theorem}[{\cite[Thm.~5.1]{grandis-mauri}}] \label{thm:normal_form}
    Every map in the category $\Box$ can be factored uniquely as a composite
    \[ (\partial_{c_1, \varepsilon'_1} \ldots \partial_{c_r, \varepsilon'_r})
       (\gamma_{b_1,\varepsilon_1} \ldots \gamma_{b_q,\varepsilon_q})
       (\sigma_{a_1} \ldots \sigma_{a_p})\text{,} \]
    where $1 \leq a_1 < \ldots < a_p$, $1 \leq b_1 \leq \ldots \leq b_q$, $b_i < b_{i+1}$ if $\varepsilon_{i} = \varepsilon_{i+1}$, and $c_1 > \ldots > c_r \geq 1$.   \qed
  \end{theorem}

\begin{definition}
 The category $\cSet$ of \emph{cubical sets} is the functor category $\fcat{\boxcat^\op}{\Set}$.
We refer to objects and morphisms of $\cSet$ as \emph{cubical sets} and \emph{cubical maps}.

The category $\cAb$ of \emph{cubical abelian groups} is the functor category $\fcat{\boxcat^\op}{\Ab}$.
\end{definition}

Given a cubical set $X$, we write $X_n$ for the value of $X$ on the object $[1]^n$ and write cubical operators on the right, e.g.~given an $n$-cube $x \in X_n$ of $X$, we write $x\face{}{1,0}$ for the $\face{}{1,0}$-face of $x$.
Likewise, for a cubical abelian group $A$, we denote the group $A([1]^n)$ by $A_n$ for $n \geq 0$ and write cubical operators on the right. 

\begin{definition}
    Let $n \geq 0$.
    \begin{itemize}
        \item The \emph{combinatorial $n$-cube} $\cube{n}$ is the representable functor $\boxcat(-, [1]^n) \from \boxcat^\op \to \Set$.
        \item The \emph{boundary of the $n$-cube} $\bd \cube{n}$ is the subobject of $\cube{n}$ defined by
        \[ \bd \cube{n} := \bigcup\limits_{\substack{j=1,\dots,n \\ \eta = 0, 1}} \im \face{}{j,\eta}. \]
        \item Given $i = 1, \dots, n$ and $\varepsilon = 0, 1$, the \emph{$(i,\varepsilon)$-open box} $\dfobox$ is the subobject of $\bd \cube{n}$ defined by
        \[ \dfobox := \bigcup\limits_{(j,\eta) \neq (i,\varepsilon)} \im \face{}{j,\eta}. \]
    \end{itemize}
\end{definition}

A large class of examples of cubical sets comes from topological spaces via the construction of a (cubical) singular complex.

\begin{example}
    Define a functor $\boxcat \to \Top$ from the box category to the category of topological spaces which sends $[1]^n$ to $[0, 1]^n$ where $[0, 1]$ is the unit interval.
    Left Kan extension along the Yoneda embedding gives the \emph{geometric realization} functor $\reali{\uvar} \from \cSet \to \Top$.
    \[ \begin{tikzcd}[column sep = large]
        \boxcat \ar[r, "{[1]^n \mapsto [0, 1]^n}"] \ar[d] & \Top \\
        \cSet \ar[ur, "\reali{\uvar}"']
    \end{tikzcd} \]
    This functor is left adjoint to the \emph{cubical singular complex} functor $\Sing{} \from \Top \to \cSet$ defined by
    \[ (\Sing{S})_n := \Top ([0, 1]^n, S). \]
\end{example}

The cubical sets arising from topological spaces satisfy an additional lifting property, making them particularly convenient for the purposes of homotopy theory.

\begin{definition}
    A cubical set $K$ is a \emph{Kan complex} if for any map $\dfobox \to K$, there exists $\cube{n} \to K$ such that the diagram
    \[ \begin{tikzcd}
        \dfobox \ar[r] \ar[d, hook] & K \\
        \cube{n} \ar[ur, dotted]
    \end{tikzcd} \]
    commutes.
\end{definition}
We write $\Kan$ for the full subcategory of $\cSet$ consisting of Kan complexes.
\begin{example}
    For any $S \in \Top$, we have that $\Sing{S}$ is a Kan complex.
    By the geometric realization and cubical singular complex adjunction, a map $\dfobox \to \Sing{S}$ corresponds to a map $\reali{\dfobox} \to S$.
    The inclusion $\reali{\dfobox} \ito \reali{\cube{n}}$ has a retraction in $\Top$.
    Pre-composing with this retraction gives a map $\reali{\cube{n}} \to S$ which restricts to the open box map $\reali{\dfobox} \to S$.
    \[ \begin{tikzcd}
        \reali{\dfobox} \ar[d, hook] \ar[r] & S \\
        \reali{\cube{n}} \ar[u, bend right]
    \end{tikzcd} \]
    This gives a suitable lift $\cube{n} \to \Sing{S}$.
\end{example}
\begin{example}[{\cite[Thm.~2.1]{tonks:cubical-kan}}]
    Every cubical abelian group, when regarded as a cubical set, is a Kan complex.
\end{example}

Define a functor $\gprod \from \boxcat \times \boxcat \to \boxcat$ by mapping $([1]^m, [1]^n)$ to $[1]^{m+n}$.
Postcomposing with the Yoneda embedding and left Kan extending gives a monoidal product on cubical sets.
\[ \begin{tikzcd}
    \boxcat \times \boxcat \ar[r] \ar[d] & \boxcat \ar[r] & \cSet \\
    \cSet \times \cSet \ar[urr, "\gprod"']
\end{tikzcd} \]
This is the \emph{geometric product} of cubical sets.
This product is biclosed. For a cubical set $X$, we write $\lhom(X, \uvar) \from \cSet \to \cSet$ and $\rhom (X, \uvar) \from \cSet \to \cSet$ for the right adjoints to the functors $\uvar \gprod X$ and $X \gprod \uvar$, respectively.

The following result gives an explicit description of cubes in the geometric product.
\begin{proposition}[{\cite[Prop.~1.24]{doherty-kapulkin-lindsey-sattler}}] \label{thm:gprod_cube}
    Let $X, Y$ be cubical sets.
    \begin{enumerate}
        \item For $k \geq 0$, the $k$-cubes of $X \gprod Y$ consists of all pairs $(x \in X_m, y \in Y_n)$ such that $m + n = k$, subject to the identification $(x\degen{}{m+1}, y) = (x, y\degen{}{1})$.
        \item For $x \in X_m$ and $y \in Y_n$, the faces, degeneracies, and connections of the $(m+n)$-cube $(x, y)$ are computed by
        \begin{align*}
            (x, y)\face{}{i,\varepsilon} &= \begin{cases}
                (x\face{}{i,\varepsilon}, y) & 1 \leq i \leq m \\
                (x, y\face{}{i-m, \varepsilon}) & m+1 \leq i \leq m+n;
            \end{cases} \\
            (x, y)\degen{}{i} &= \begin{cases}
                (x\degen{}{i}, y) & 1 \leq i \leq m+1 \\
                (x, y\degen{}{i-m}) & m+1 \leq i \leq m+n;
            \end{cases} \\
            (x, y)\conn{}{i,\varepsilon} &= \begin{cases}
                (x\conn{}{i,\varepsilon}, y) & 1 \leq i \leq m \\
                (x, y\conn{}{i-m,\varepsilon}) & m+1 \leq i \leq n.
            \end{cases}
        \end{align*}
    \end{enumerate} \noproof
\end{proposition}

Using the geometric product, we may define the notion of homotopy for maps of cubical sets.
\begin{definition}
    Given cubical maps $f, g \from X \to Y$, a \emph{homotopy} from $f$ to $g$ is a map $h \from X \gprod \cube{1} \to Y$ such that the diagram
    \[ \begin{tikzcd}
        X \gprod \cube{0} \ar[d, left, "\face{}{1,0}"'] \ar[dr, "f"] & \\
        X \gprod \cube{1} \ar[r, "h"] & Y \\
        X \gprod \cube{0} \ar[u, left, "\face{}{1,1}"] \ar[ur, "g"'] & 
    \end{tikzcd} \]
    commutes.
\end{definition}

In order to define homotopy and reduced homology groups, we work with pointed cubical sets and based maps between them.
\begin{definition} 
    \begin{enumerate} 
    \item A \emph{pointed cubical set} is an object $(X,x)$ in the slice category $\cube{0} \slice \cSet$ under $\cube{0}$, where $x$ denotes the map $\cube{0} \to X$.
    \item A \emph{based map} is a morphism $(X,x) \to (Y,y)$ in $\cube{0} \slice \cSet$.
    \item A \emph{based homotopy} between two based maps $f, g \from (X,x) \to (Y,y)$ is a homotopy $h \from X \gprod \cube{1} \to Y$ between the underlying maps $f, g \from X \to Y$ such that the restriction of $h$ to the cubical subset $\{ x \} \gprod \cube{1} \subseteq X \gprod \cube{1}$ is constant at the point $y \in Y$.
    That is, there is a factorization:
    \[ \begin{tikzcd}[column sep = large]
        \{ x \} \gprod \cube{1} \ar[r, dotted, "\restr{h}{\{ x \} \gprod \cube{1}}"] \ar[d, hook, "x \gprod \cube{1}"'] & \{ y \} \ar[d, hook, "y"] \\
        X \gprod \cube{1} \ar[r, "h"] & Y
    \end{tikzcd} \] 
    \end{enumerate}
\end{definition}
We write $\cSet_*$ for the category of pointed cubical sets.
We occasionally write a pointed cubical set $(X,x)$ as simply $X$, omitting the base point to aid readability.

%% file: 2-homotopy.tex
\section{Homotopy groups of Kan complexes} \label{sec:homotopy}

In this section, we briefly summarize the relevant definitions of loop spaces and homotopy groups of cubical sets from \cite{carranza-kapulkin:homotopy-cubical}.

We begin by recalling the definition of connected components of a Kan complex.
\begin{definition}
    The \emph{connected components} functor $\pi_0 \from \cSet_* \to \Set_*$ is the functor which evaluates the colimit of a pointed cubical set $X$, regarded as a diagram $\boxcat^\op \to \Set_*$. 
\end{definition}
If $X$ is a pointed Kan complex then $\pi_0 X$ is exactly the (pointed) set of (unbased) homotopy classes of maps $\cube{0} \to X$.
In analogy with spaces, the $n$-th homotopy group of a Kan complex is the set of connected components of its $n$-th loop space, which we move towards defining.

\begin{notation}
Given a cubical set $X$, a 0-cube $x \in X_0$, and $n \geq 0$, we write $x\ndegen$ for the $n$-cube $x\degen{1}{1}\degen{2}{1}\dots\degen{n}{1}$, i.e., the $n$-cube that is degenerate at the $0$-cube $x$.
\end{notation}

\begin{definition}
    For a pointed Kan complex $(X,x)$, the \emph{loop space} $\loopsp(X,x)$ of $(X, x)$ is the pullback
    \[ \begin{tikzcd}
        \loopsp(X,x) \ar[r] \ar[d] \ar[rd, phantom, "\ulcorner" very near start] & \rhom(\cube{1}, X) \ar[d, "{(\face{*}{1,0}, \face{*}{1,1})}"] \\
        \cube{0} \ar[r, "{(x, x)}"] & X \times X
    \end{tikzcd}  \]
    with distinguished basepoint $x\degen{}{1} \in \loopsp(X,x)$.
\end{definition}
Explicitly, an $n$-cube of $\loopsp(X,x)$ is an $(n+1)$-cube of $X$ whose $\face{}{1,0}$- and $\face{}{1,1}$-faces are $x\ndegen$.
If $X$ is a Kan complex then so is $\loopsp X$ \cite[Cor.~3.4]{carranza-kapulkin:homotopy-cubical}.

\begin{definition} \label{def:concat}
    Let $(X,x)$ be a pointed Kan complex.
    Given $u, v \in (\loopsp (X,x))_0$,
    \begin{enumerate}
        \item a \emph{concatenation square} for $u$ and $v$ is a filler $\cube{2} \to X$ for the map $f \from \obox{2}{2,0} \to X$ defined by
        \[ \begin{array}{l l}
            f\face{}{1,0} := u & f\face{}{1,1} := x\degen{}{1} \\
            & f\face{}{2,1} := v
        \end{array} \qquad \begin{tikzcd}
            x \ar[d, "u"'] & x \ar[d, equal] \\
            x \ar[r, "v"] & x
        \end{tikzcd} \]
        \item a \emph{concatenation} of $u$ and $v$ is a 0-cube of $\loopsp(X,x)$ which is the $\face{}{2,0}$-face of a concatenation square for $u$ and $v$.
        \[ \begin{tikzcd}
            x \ar[d, "u"'] \ar[r, dotted] & x \ar[d, equal] \\
            x \ar[r, "v"] & x
        \end{tikzcd} \]
    \end{enumerate}
\end{definition}
\begin{theorem}[{\cite[Thm.~3.11]{carranza-kapulkin:homotopy-cubical}}]
    Concatenation induces a well-defined group structure 
    \[ \pi_0 \loopsp X \times \pi_0 \loopsp X \to \pi_0 \loopsp X \] 
    on connected components of $\loopsp X$. \noproof
\end{theorem}
\begin{definition}
    The \emph{fundamental group} $\pi_1 (X,x)$ of a pointed Kan complex $(X,x)$ is the group $\pi_0 \loopsp (X,x)$ under concatenation.
\end{definition}
\begin{remark} \label{rmk:conc}
    As stated in \cite[Rmk.~3.12]{carranza-kapulkin:homotopy-cubical}, this definition of the fundamental group relies on negative connections.
    This definition may be modified to work with positive connections using a different definition of concatenation squares where either the $\face{}{1,0}$- or $\face{}{1,1}$-face is $x\degen{}{1}$.
\end{remark}
Higher homotopy groups are defined by iterating the loop space construction.
\begin{definition}
    For a pointed Kan complex $(X,x)$ and $n \geq 0$, 
    \begin{enumerate}
        \item the \emph{$n$-th loop space} $\loopsp[n] (X,x)$ of $(X,x)$ is defined inductively by
        \[ \loopsp[n] (X,x) := \begin{cases}
            (X,x) & n = 0 \\
            \loopsp(\loopsp[n-1] (X, x), x\ndegen[n-1]) & n > 0.
        \end{cases} \]
        \item the \emph{$n$-th homotopy group} $\pi_n(X,x)$ of $(X,x)$ is defined by
        \[ \pi_n(X,x) := \pi_0 \loopsp[n](X,x\ndegen). \]
    \end{enumerate}
\end{definition}
Note that there is a notion of concatenation (which is well-defined up to homotopy) for $n$-cubes whose faces are not all degenerate (cf.~\cite[pg.~27-28]{carranza-kapulkin:homotopy-cubical}).
We record one such notion, horizontal concatenation of squares, for later use.
\begin{definition}[{\cite[pg.~28]{carranza-kapulkin:homotopy-cubical}}]
    Let $X$ be a cubical set and $u, v \in X_2$ be such that $u\face{}{1,1} = v\face{}{1,0}$.
    \begin{enumerate}
    \item A \emph{horizontal concatenation cube} for $u$ and $v$ is a filler $\cube{3} \to X$ for the map $f \from \obox{3}{2,0} \to X$ defined by
    \[ \begin{array}{l l}
        f\face{}{1,0} := u & f\face{}{1,1} := v\face{}{1,1}\degen{}{1} \\
        & f\face{}{2,1} := v \\
        f\face{}{3,0} = \cpsq{u\face{}{2,0}, v\face{}{2,0}} & f\face{}{3,1} = \cpsq{u\face{}{2,1}, v\face{}{2,1}}
    \end{array} \qquad \begin{tikzcd}
        \cdot \ar[rr] \ar[rd] \ar[ddd] \ar[rdddd, phantom, "u" description] && \cdot \ar[rd, equal] \ar[ddd] & \\
        & \cdot \ar[rr, crossing over] \ar[rrddd, "v" description, phantom, xshift=-3ex, yshift=1ex] && \cdot \ar[ddd] \\
        \\
        \cdot \ar[rr] \ar[rd] && \cdot \ar[rd, equal] & \\
        & \cdot \ar[rr] \ar[from=uuu, crossing over] && \cdot
    \end{tikzcd} \]
    where $\cpsq{u\face{}{2,0}$, $v\face{}{2,1}}$ and $\cpsq{u\face{}{2,1}$, $v\face{}{2,1}}$ are concatenation squares for the pairs $u\face{}{2,0}, v\face{}{2,0}$ and $u\face{}{2,1}, v\face{}{2,1}$, respectively (note that concatenation squares as in \cref{def:concat} are well-defined for these pairs despite not being 0-cubes of $\loopsp X$).
    \item A \emph{horizontal concatenation} for $u$ and $v$ is a 2-cube of $X$ which is the $\face{}{2,0}$-face of a horizontal concatenation cube for $u$ and $v$.
    \end{enumerate}
\end{definition}

The notion of homotopy groups gives a notion of $n$-connected maps and $n$-connected Kan complexes.
\begin{definition}
    Let $n \geq 1$.
    \begin{enumerate}
        \item A pointed map $f \from X \to Y$ between pointed Kan complexes is \emph{$n$-connected} if the map $\pi_k (f) \from \pi_k (X) \to \pi_k (Y)$ is an isomorphism for $0 \leq k \leq n$, and an epimorphism for $k = n+1$.
        \item A pointed Kan complex $X$ is \emph{$n$-connected} if the unique map $X \to \cube{0}$ is $n$-connected.
    \end{enumerate}
\end{definition}

Thus a pointed Kan complex $X$ is $n$-connected exactly when $\pi_k(X)$ is trivial for $k \leq n$.
In particular, a Kan complex is $0$-connected if and only if it is connected.

%% file: 3-homology.tex
\section{Cubical homology} \label{sec:homology}

We now recall the construction and main properties of cubical homology, with a special emphasis on the results presented in \cite{barcelo-greene-jarrah-welker}.
We keep our presentation self-contained, starting with the definition of the homology of a (bounded) chain complex with integer coefficients.

\begin{definition}
    A (bounded) \emph{chain complex} $C_\bullet$ (over $\Z$) consists of a set $\{ C_n \mid n \in \N \}$ of abelian groups and, for $n \geq 1$, a group homomorphism $\cmap[n] \from C_n \to C_{n-1}$ such that $\cmap[n-1]\cmap[n] = 0$.
    
    We write $\Ch$ for the category of chain complexes.
\end{definition}

As alluded to in the above definition, throughout this and subsequent section, we will be working with homology with integer coefficients.
With that in mind, we will be omitting the coefficient group from our notation for homology, writing simply $H_n X$ for the $n$-th homology group of $X$ or $H_* X$ for the corresponding graded abelian group.

\begin{definition}
    The \emph{homology} functor $H_* \from \Ch \to \fcat{\N}{\Ab}$ is defined by
    \[ H_n C := \begin{cases}
        C_0 / \im \cmap[1] & n = 0 \\
        \ker \cmap[n] / \im \cmap[n+1] & n \geq 1.
    \end{cases} \]
\end{definition}

We now move to define chain complexes associated to cubical abelian groups via the Moore complex.
This will be done by first defining a slightly bigger chain complex and then quotienting by its degenerate subcomplex.

Given a cubical abelian group $A$ and $n \geq 1$, define the \emph{alternating face map} $\cmap[n] \from A_n \to A_{n-1}$ by
\[ \cmap[n](x) = \sum\limits_{\substack{i=1, \dots, n \\ \varepsilon = 0, 1}} (-1)^{i + \varepsilon} x \face{}{i,\varepsilon}.  \]
\begin{proposition} \label{thm:cab_ch}
    For a cubical abelian group $A$, the set $A_\bullet := \{ A_n \mid n \in \N \}$ with the alternating face maps $\cmap[n] \from A_n \to A_{n-1}$ forms a chain complex.
\end{proposition}

\begin{proof}
  This is a standard argument.
    Given $x \in A_n$, we have
    \begin{align*}
        \cmap[n-1] \cmap[n] (x) &= \sum\limits_{\substack{i = 1, \dots, n \\ j = 1, \dots, n-1 \\ \varepsilon, \varepsilon' = 0, 1}} (-1)^{i+j+\varepsilon + \varepsilon'} x \face{}{i,\varepsilon} \face{}{j,\varepsilon'} \\
        &= \sum\limits_{\substack{i = 2, \dots, n \\ j = 1, \dots, i-1 \\ \varepsilon, \varepsilon' = 0, 1}} (-1)^{i+j+\varepsilon+\varepsilon'} x \face{}{i,\varepsilon} \face{}{j,\varepsilon'} + \sum\limits_{\substack{i = 1, \dots, n-1 \\ j = i, \dots, n-1 \\ \varepsilon, \varepsilon' = 0, 1}} (-1)^{i+j+\varepsilon+\varepsilon'} x \face{}{i,\varepsilon} \face{}{j,\varepsilon'} \\
        &= \sum\limits_{\substack{i = 2, \dots, n \\ j = 1, \dots, i-1 \\ \varepsilon, \varepsilon' = 0, 1}} (-1)^{i+j+\varepsilon+\varepsilon'} x \face{}{i,\varepsilon} \face{}{j,\varepsilon'}  + \sum\limits_{\substack{i = 1, \dots, n-1 \\ j = i, \dots, n-1 \\ \varepsilon, \varepsilon' = 0, 1}} (-1)^{i+j+\varepsilon+\varepsilon'} x \face{}{j+1,\varepsilon'} \face{}{i,\varepsilon} \\
        &= \sum\limits_{\substack{j=1, \dots, n-1 \\ i=j+1, \dots, n \\ \varepsilon, \varepsilon' = 0, 1}} (-1)^{i+j+\varepsilon+\varepsilon'} x \face{}{i,\varepsilon}\face{}{j,\varepsilon'} + \sum\limits_{\substack{i=1, \dots, n-1 \\ j=i+1, \dots, n \\ \varepsilon, \varepsilon' = 0, 1}} (-1)^{i+j-1+\varepsilon+\varepsilon'} x\face{}{j,\varepsilon'} \face{}{i,\varepsilon} \\
        &= \sum\limits_{\substack{j=1, \dots, n-1 \\ i=j+1, \dots, n}} \left( (-1)^{i+j} x \face{}{i,0}\face{}{j,0} + (-1)^{i+j+1} x\face{}{i,0}\face{}{j,1} + (-1)^{i+j+1} x\face{}{i,1}\face{}{j,0} + (-1)^{i+j} x\face{}{i,1} \face{}{j,1} \right) \\
        &\quad + \sum\limits_{\substack{i=1, \dots, n-1 \\ j=i+1, \dots, n}} \left( (-1)^{i+j-1} x\face{}{j,0} \face{}{i,0} + (-1)^{i+j} x\face{}{j,0}\face{}{i,1} + (-1)^{i+j} x\face{}{j,1}\face{}{i,0} + (-1)^{i+j-1} x\face{}{j,1}\face{}{i,1} \right) \\
        &= 0.
    \end{align*}
\end{proof}

Given a morphism $f \from A \to B$ of cubical abelian groups, naturality of $f$ implies the square
\[ \begin{tikzcd}
    A_n \ar[r, "{\cmap[n]}"] \ar[d, "f_n"'] & A_{n-1} \ar[d, "f_{n-1}"] \\
    B_n \ar[r, "{\cmap[n]}"] & B_{n-1}
\end{tikzcd} \]
commutes for all $n \geq 1$.
This gives a functor $\cAb \to \Ch$ which maps a cubical abelian group $A$ to the chain complex $A_\bullet$.

\begin{definition} \label{def:degenerate-subcomplex}
    Let $A$ be a cubical abelian group.
    \begin{enumerate}
        \item The \emph{subcomplex of degeneracies} $\D[\degen{}{}] A_\bullet$ of $A_\bullet$ is the chain subcomplex generated by degeneracies.
        That is,
        \[ \D[\degen{}{}] A_n = \left\{ \sum\limits_{i=1}^n x\degen{}{i} \in A_n \mid x \in A_{n-1} \right\}. \]
        \item The \emph{subcomplex of negative connections} $\D[\conn{}{}-]$ of $A_\bullet$ is the chain subcomplex generated by negative connections.
        That is,
        \[ \D[\conn{}{}-] A_n = \left\{ \sum\limits_{i=1}^{n-1} x\conn{}{i,0} \in A_n \mid x \in A_{n-1} \right\}. \]
        \item The \emph{subcomplex of positive connections} $\D[\conn{}{}+]$ of $A_\bullet$ is the chain subcomplex generated by positive connections. 
        That is,
        \[ \D[\conn{}{}+] A_n = \left\{ \sum\limits_{i=1}^{n-1} x\conn{}{i,1} \in A_n \mid x \in A_{n-1} \right\}. \]
        \item The \emph{degenerate subcomplex} $\D A_\bullet$ of $A_\bullet$ is defined by
        \[ \D A_n := \D[\degen{}{}] A_n + \D[\conn{}{}-] A_n = \left\{ \sum\limits_{i=1}^n x\degen{}{i} + \sum\limits_{i=1}^{n-1} y\conn{}{n,0} \mid x, y \in A_{n-1} \right\}. \] 
    \end{enumerate}
\end{definition}

Note that $\D[\degen{}{}] A_0$ is the trivial subgroup, as are $\D[\conn{}{}-] A_n, \D[\conn{}{}+] A_n$ if $n = 0, 1$.

\begin{proposition}[cf.~{\cite[Cor.~3.2]{barcelo-greene-jarrah-welker}}]
    For a cubical abelian group $A$, 
    \begin{enumerate}
        \item the subcomplex of degeneracies is a chain complex;
        \item the subcomplex of negative connections is a chain complex;
        \item the subcomplex of positive connections is a chain complex.
    \end{enumerate}
    That is, for $n \geq 1$, and $\alpha = \degen{}{}, \conn{}{}-, \conn{}{}+$, the restriction of the alternating face map $\cmap[n] \from A_n \to A_{n-1}$ to $\D[\alpha] A_n \subseteq A_n$ takes values in $\D[\alpha] A_{n-1}$, giving a factorization as in the following diagram.
    \[ \begin{tikzcd}
        {\D[\alpha] A_n} \ar[d, hook] \ar[r, dotted, "{\restr{\cmap[n]}{\D[\alpha] A_n}}"] & {\D[\alpha] A_{n-1}} \ar[d, hook] \\
        A_n \ar[r, "{\cmap[n]}"] & A_{n-1}
    \end{tikzcd} \]
\end{proposition}
\begin{proof}
    We prove (1), and note that (2) and (3) are proved in \cite[Cor.~3.2]{barcelo-greene-jarrah-welker}.
    Given $x \in A_{n-1}$ and $i = 1, \dots, n$, cubical identities imply
    \begin{align*}
        \cmap[n] (x\degen{}{i}) &= \sum\limits_{\substack{j=1, \dots, n \\ \varepsilon = 0, 1}} (-1)^{j + \varepsilon} x \degen{}{i} \face{}{j,\varepsilon} \\
        &= \sum\limits_{\substack{j=1, \dots, i-1 \\ \varepsilon = 0, 1}} (-1)^{j + \varepsilon} x \face{}{j,\varepsilon} \degen{}{i-1} + (-1)^i x + (-1)^{i+1} x + \sum\limits_{\substack{j=i+1, \dots, n \\ \varepsilon = 0, 1}} (-1)^{j + \varepsilon} x\face{}{j-1,\varepsilon}\degen{}{i} \\
        &= \sum\limits_{\substack{j=1, \dots, i-1 \\ \varepsilon = 0, 1}} (-1)^{j + \varepsilon} x \face{}{j,\varepsilon} \degen{}{i-1} + \sum\limits_{\substack{j=i+1, \dots, n \\ \varepsilon = 0, 1}} (-1)^{j + \varepsilon} x\face{}{j-1,\varepsilon}\degen{}{i}.
    \end{align*}
    By definition, this element is in $\D[\degen{}{}] A_{n-1}$. 
\end{proof}
\begin{definition}
    The \emph{Moore complex} functor $\C \from \cAb \to \Ch$ is the functor which sends $A \in \cAb$ to the quotient chain complex $A_\bullet / \D A_\bullet$.
\end{definition}

\begin{remark} \label{rmk:def}
    One may also define the Moore complex of a cubical abelian group by quotienting by degeneracies, by degeneracies and negative connections, or by degeneracies and both connections.
    These definitions all give isomorphic homology functors \cite[Cor.~3.10]{barcelo-greene-jarrah-welker}.
    However, the argument we present for the Hurewicz theorem (\cref{thm:hurewicz1,thm:hurewicz}) relies on quotienting by one type of connection (\cref{rmk:counter-example}).
    Any definitions and proofs which rely on a chosen definition for the Moore complex will be commented on explicitly (\cref{rmk:iso,rmk:conc,rmk:def,rmk:a_n_mul,rmk:na_plus_da}).
\end{remark}

Recall the \emph{pointed free abelian group} functor $F \from \Set_* \to \Ab$ is defined by
\[ F(S, s) := \Z^{\oplus |S|} / \langle s \rangle , \]
where $|S|$ denotes the cardinality of $S$ and $\langle s \rangle$ denotes the subgroup generated by $s$.
This induces a functor $F \from \cSet_* \to \cAb$ by post-composition.
That is, the group $(F(X,x))_n$ is the pointed free abelian group $F(X_n, x\ndegen)$.
The pointed free abelian group functor is left adjoint to the forgetful functor $U \from \cAb \to \cSet_*$ defined by $A \mapsto (A, 0)$.

\begin{definition}
    The \emph{reduced homology} functor $\pH \from \cSet_* \to \fcat{\N}{\Ab}$ is the composite
    \[ \cSet_* \xrightarrow{F} \cAb \xrightarrow{\C} \Ch \xrightarrow{H_\bullet} \fcat{\N}{\Ab}. \]
\end{definition}
\begin{remark} \label{rmk:iso}
    The subcomplexes of negative and positive connections are chain homotopy equivalent to 0 by \cite[Cor.~3.10]{barcelo-greene-jarrah-welker}.
    This implies each map in the diagram
    \[ \begin{tikzcd}[sep = small]
        {} & {\D[\degen{}{}] A_\bullet + \D[\conn{}{}-] A_\bullet} \ar[rd, "\sim"] & {} \\
        {\D[\degen{}{}] A_\bullet} \ar[ur, "\sim"] \ar[dr, "\sim"] && {\D[\degen{}{}] A_\bullet + \D[\conn{}{}-] A_\bullet + \D[\conn{}{}+] A_\bullet} \\
        {} & {\D[\degen{}{}] A_\bullet + \D[\conn{}{}+] A_\bullet} \ar[ur, "\sim"] & {}
    \end{tikzcd} \]
    is a chain homotopy equivalence.
    This induces a diagram
    \[ \begin{tikzcd}[sep = small]
        {} & {A_\bullet / (\D[\degen{}{}] A_\bullet + \D[\conn{}{}-] A_\bullet)} \ar[rd, "\sim"] & {} \\
        {A_\bullet / \D[\degen{}{}] A_\bullet} \ar[ur, "\sim"] \ar[dr, "\sim"] && {A_\bullet / (\D[\degen{}{}] A_\bullet + \D[\conn{}{}-] A_\bullet + \D[\conn{}{}+] A_\bullet)} \\
        {} & {A_\bullet / (\D[\degen{}{}] A_\bullet + \D[\conn{}{}+] A_\bullet)} \ar[ur, "\sim"] & {}
    \end{tikzcd} \]
    in which every map is, again, a chain homotopy equivalence.
    In particular, these maps induce isomorphisms on homology, thus all definitions of the Moore complex given in \cref{rmk:def} give isomorphic homology groups.
\end{remark}

\begin{remark}
  When defining (reduced) homology of simplicial sets, one considers a similar composite
      \[ \sSet_* \rightarrow \mathsf{sAb} \rightarrow \Ch \rightarrow \fcat{\N}{\Ab}\text{,} \]
      however the middle functor $\mathsf{sAb} \rightarrow \Ch$ can be taken to either quotient by degeneracies or not.
      In the case of cubical homology, we do not have this choice --- if we do not quotient by degeneracies, we would get that $\pH_1(\Box^0 \sqcup \Box^0)$ is $\mathbb{Z}$.
\end{remark}

We show that homology is a homotopy functor i.e.~it sends homotopies to equalities.
\begin{theorem} \label{thm:homology_htpy_eq}
    Let $f, g \from (X, x) \to (Y, y)$ be pointed cubical maps.
    If $h \from X \gprod \cube{1} \to Y$ is a based homotopy from $f$ to $g$ then $\pH_* f = \pH_* g$. 
\end{theorem}
\begin{proof}
    It suffices to construct a chain homotopy between chain maps $f, g \from \C FX \to \C FY$.
    For $n \geq 0$, define $\alpha_n \from FX_n \to FY_{n+1}$ on basis elements by
    \[ \alpha_n(z) = (-1)^{n+1} h(z, \id[\cube{1}]). \]
    Since $h$ is a based homotopy, $\alpha_n(x) = \pm y$, thus this map is well-defined.
    If $z$ is a degeneracy or negative connection then $\alpha_n(z)$ is as well by \cref{thm:gprod_cube}.
    Thus, $\alpha_n$ induces a map $FX/\D FX \to FY/\D FY$.

    It remains to show $\alpha_{n-1} \cmap + \cmap \alpha_n = f - g$.
    Verifying this equality on basis elements, given $z \in X_n$, we have
    \begin{align*}
        (\alpha_{n-1} \cmap + \cmap \alpha_n)(z) &= \sum\limits_{\substack{i=1, \dots, n \\ \varepsilon = 0, 1}}(-1)^{i + \varepsilon}(-1)^{n} h(z\face{}{i,\varepsilon}, \id[\cube{1}]) + \sum\limits_{\substack{i=1, \dots, n+1 \\ \varepsilon = 0, 1}} (-1)^{i+\varepsilon} (-1)^{n+1} h(z, \id[\cube{1}]) \face{}{i,\varepsilon} \\
        &= \sum\limits_{\substack{i=1, \dots, n \\ \varepsilon = 0, 1}}(-1)^{i + \varepsilon}(-1)^{n} h(z\face{}{i,\varepsilon}, \id[\cube{1}]) + \sum\limits_{\substack{i=1, \dots, n \\ \varepsilon = 0, 1}} (-1)^{i+\varepsilon} (-1)^{n+1} h(z\face{}{i,\varepsilon}, \id[\cube{1}]) \\
        &\qquad + (-1)^{2n+2} h(z, \face{}{1,0}) + (-1)^{2n+3} h(z, \face{}{1,1}) \tag{$\ast$} \\
        &= h(z, \face{}{1,0}) - h(z, \face{}{1,1}) \\
        &= f(z) - g(z).
    \end{align*}
    where $(\ast)$ follows from \cref{thm:gprod_cube}.
\end{proof}
\begin{corollary} \label{thm:homology_htpy_equiv_iso}
    If $f \from X \to Y$ is a based homotopy equivalence then $\pH_* f \from \pH_* X \to \pH_* Y$ is an isomorphism of graded abelian groups.
    \noproof
\end{corollary}

%% file: 4-hurewicz.tex
\section{Hurewicz Theorem} \label{sec:hurewicz}

In this final section of the paper, we will construct the Hurewicz homomorphism (\cref{def:Hurewicz-map}) and prove the Hurewicz theorem (\cref{thm:hurewicz1,thm:hurewicz}).
The preliminary steps of our approach will follow those found in the classical proofs for simplicial homology \cite{weibel:homological-algebra,goerss-jardine}.

To this end, we begin by showing that for a cubical abelian group, its $n$-th homotopy group and its $n$-th homology group coincide (\cref{thm:pi_n_h_n}).
The key technical ingredient of this is a theorem (cf.~\cite[Thm.~III.2.1]{goerss-jardine} and \cite[Lem.~8.3.7]{weibel:homological-algebra} in the simplicial setting) asserting that for any cubical abelian group $A$, its $n$-th abelian group $A_n$ can be written as a direct sum
\[A_n \cong NA_n \oplus \D A_n\]
of the \emph{normalized subcomplex} (\cref{def:normalized-subcomplex}) and the \emph{degenerate subcomplex} (\cref{def:degenerate-subcomplex}).
This is done in \cref{thm:na_oplus_da}, and the proof of \cref{thm:pi_n_h_n} follows.

From this point on, our approach takes a different turn than the proofs found in the literature for the simplicial case, as it is based on the notion of loop space of a (pointed) Kan complex and the interaction between the loop space and the free-forgetful adjunction between cubical sets and cubical abelian groups.

\begin{theorem} \label{thm:pi_n_h_n}
    For a cubical abelian group $A$, we have an isomorphism
    \[ \pi_n(UA, 0) \cong \pH_n(A, 0) \]
    natural in $A$.
\end{theorem}
To prove \cref{thm:pi_n_h_n}, we establish some auxiliary results and definitions.
\begin{proposition} \label{thm:pi_n_a_n_mul}
    Let $A$ be a cubical abelian group and $n \geq 1$.
    Given $[x], [y] \in \pi_n(UA, 0)$, we have that
    \[ [x + y] = [x][y]. \]
\end{proposition}
\begin{proof}
    The $(n+1)$-cube $x\conn{}{n,0} + y\degen{}{n+1}$ is a concatenation square in $\loopsp[n-1] UA$ witnessing $[x + y]$ as a concatenation of $[x]$ and $[y]$.
\end{proof}
\begin{corollary} \label{thm:pi_n_a_n_comm}
    For any cubical abelian group $A$ and $n \geq 1$, the group $\pi_n(UA, 0)$ is abelian. \noproof
\end{corollary}
\begin{remark} \label{rmk:a_n_mul}
    When using a definition of concatenation mentioned in \cref{rmk:conc}, an analogous proof of \cref{thm:pi_n_a_n_mul} using positive connections applies.
\end{remark}

\begin{definition} \label{def:normalized-subcomplex}
    For a cubical abelian group $A$ and $n \geq 0$, the \emph{normalized subcomplex} $NA_\bullet$ of $A_\bullet$ is the intersection of kernels
    \[ (NA)_n := \bigcap\limits_{(i,\varepsilon) \neq (n,0)} \ker \face{}{i,\varepsilon} \subseteq A_n. \]
\end{definition}
In particular, $NA_0 = A_0$.
\begin{proposition}
    For a cubical abelian group $A$, the normalized subcomplex of $A_\bullet$ is indeed a chain complex.
    That is, for $n \geq 0$, the restriction of the alternating face map $\cmap[n] \from A_n \to A_{n-1}$ to $NA_n \subseteq A_n$ takes values in $NA_{n-1}$, giving a factorization as in the following diagram.
    \[ \begin{tikzcd}
        NA_n \ar[d, hook] \ar[r, dotted, "{\restr{\cmap[n]}{NA_n}}"] & NA_{n-1} \ar[d, hook] \\
        A_n \ar[r, "{\cmap[n]}"] & A_{n-1}
    \end{tikzcd} \]
\end{proposition}
\begin{proof}
    Follows by the cubical identity $\partial_{j, \varepsilon'} \partial_{i, \varepsilon} = \partial_{i+1, \varepsilon} \partial_{j, \varepsilon'}$.
\end{proof}

For a cubical abelian group $A$ and $n \geq 1$, by definition of the $n$-th loop space, there is a natural bijection
\[ \loopsp[n](UA, 0)_0 \cong \ker \restr{\face{}{n,0}}{NA_n} \]
between 0-cubes of the $n$-th loop space of $(UA, 0)$ and $n$-cubes in the kernel $\ker \face{}{n,0} \subseteq NA_n$ of the face map $\face{}{n,0} \from A_n \to A_{n-1}$ restricted to $NA_n \subseteq A_n$.
\begin{lemma} \label{thm:pi_n_h_n_na}
    For a cubical abelian group $A$ and $n \geq 1$, the bijection
    \[ \loopsp[n](UA, 0)_0 \cong \ker \restr{\face{}{n,0}}{NA_n} \]
    induces a well-defined group isomorphism
    \[ \pi_n(UA, 0) \cong H_n(NA_\bullet), \]
    where $H_n (NA_\bullet)$ denotes the $n$-th homology group of the chain complex $NA_\bullet$.
\end{lemma}
\begin{proof}
    This map is well-defined as, given a 1-cube $u \in \loopsp[n](UA, 0)_1$ witnessing an equality $[x] = [y]$ in $\pi_n UA$, we may regard $u$ as an $(n+1)$-cube $u \in A_{n+1}$ where $u\face{}{n,0} = x$, $u\face{}{n,1} = y$, and all other faces of $u$ are 0.
    The $(n+1)$-cube $u - y\degen{}{n+1} \in A_{n+1}$ is an element of $NA_{n+1}$ whose $\face{}{n,0}$-face is $x - y$.
    This map is a group homomorphism by \cref{thm:pi_n_a_n_mul}.

    By definition, the kernel of this map is the set of elements $[x] \in \pi_n(UA, 0)$ such that $[x] = [0\ndegen]$.
    Hence, this map is injective.
    Surjectivity follows since the map $\loopsp[n](UA, 0)_0 \to \ker \face{}{n,0}$ is surjective.
\end{proof}

\begin{theorem}\label{thm:na_oplus_da}
    For a cubical abelian group $A$ and $n \geq 0$, we have
    \[ A_n \cong NA_n \oplus \D A_n. \]
\end{theorem}
\begin{proof}
For $i \in \{ 0, \dots, n \}$, define subgroups $B_i A_n, C_i A_n \subseteq A_n$ by
\begin{align*}
    B_i A_n &:= \bigcap\limits_{j=1}^{n-i} \ker \face{}{j,1} \\
    C_i A_n &:= B_0 A_n \cap \left( \bigcap\limits_{j=1}^{n-i} \ker \face{}{j,0} \right),
\end{align*}
where $B_n A_n$ denotes the entire group $A_n$ and $C_n A_n$ denotes the subgroup $B_0 A_n$.
In particular, $B_0 A_0 = C_0 A_0 = A_0$ and $C_1 A_n = NA_n$.

For $i \in \{ 1, \dots, n-1 \}$, cubical identities imply the $\face{}{n-i,0}$-face map restricts to a split epimorphism $\restr{\face{}{n-i,0}}{C_{i+1} A_n} \from C_{i+1} A_n \to C_i A_{n-1}$ with pre-inverse $\restr{\conn{}{n-i,0}}{C_i A_{n-1}} \from C_i A_{n-1} \to C_{i+1} A_n$.
With this, the sequence
\[
    0 \xrightarrow{\phantom{\face{}{n-i,0}}} C_{i} A_n \xrightarrow{\phantom{\face{}{n-i,0}}} C_{i+1} A_n \xrightarrow{\face{}{n-i,0}} C_i A_{n-1} \xrightarrow{\phantom{\face{}{n-i,0}}} 0
\]
is a short exact sequence which is right split.
By the splitting lemma, $C_{i+1} A_n \cong C_i A_n \oplus C_i A_{n-1}$.
An induction argument combined with \cref{thm:normal_form} gives
\[ B_0 A_n = C_n A_n \cong NA_n \oplus \left(\bigoplus\limits_{i=1}^{n-1} NA_i^{\oplus |\boxcat_{\conn{}{}-}([1]^n, [1]^i)|} \right), \]
where $|\boxcat_{\conn{}{}-}([1]^n, [1]^i)|$ is the number of maps from $[1]^n \to [1]^i$ in $\boxcat$ generated by negative connection maps.
(In particular, \cref{thm:normal_form} allows us to identify the inclusion given abstractly by the splitting lemma with the explicit construction taking $x$ in the summand corresponding to $\alpha \in \boxcat_{\conn{}{}-}([1]^n, [1]^i)$ to $x \alpha$.)
An explicit formula for this is
\[ |\boxcat_{\conn{}{}-}([1]^n, [1]^i)| = \begin{pmatrix} n-1 \\ n-i \end{pmatrix}. \]

For $i \in \{ 0, \dots, n-1 \}$, cubical identities imply the $\face{}{n-i,1}$-face map restricts to a split epimorphism $\restr{\face{}{n-i,1}}{B_{i+1} A_n} \from B_{i+1} A_n \to B_i A_{n-1}$ with pre-inverse $\restr{\degen{}{n-i}}{B_i A_{n-1}} \from B_i A_{n-1} \to B_{i+1} A_n$.
This implies the sequence
\[
    0 \xrightarrow{\phantom{\face{}{n-i,1}}} B_{i} A_n \xrightarrow{\phantom{\face{}{n-i,1}}} B_{i+1} A_n \xrightarrow{\face{}{n-i,1}} B_i A_{n-1} \xrightarrow{\phantom{\face{}{n-i,1}}} 0
\]
is a split short exact sequence, thus $B_{i+1} A_n \cong B_i A_n \oplus B_i A_{n-1}$.
By induction and \cref{thm:normal_form},
\[ A_n = B_n A_n \cong \bigoplus\limits_{i=1}^n NA_i^{\oplus |\boxcat_{\degen{}{}, \conn{}{}-}([1]^n, [1]^i)|}, \]
where $|\boxcat_{\degen{}{}, \conn{}{}-}([1]^n, [1]^i)|$ is the number of maps $[1]^n \to [1]^i$ in $\boxcat$ generated by degeneracy and negative connection maps.
(Once again, \cref{thm:normal_form} allows us to identify the inclusion given abstractly by the splitting lemma with the explicit construction taking $x$ in the summand corresponding to $\alpha \in \boxcat_{\degen{}{}, \conn{}{}-}([1]^n, [1]^i)$ to $x \alpha$.)
An explicit formula for this is
\[ |\boxcat_{\degen{}{}, \conn{}{}-}([1]^n, [1]^i)| = \sum\limits_{j=0}^{n-i} \begin{pmatrix} n \\ j \end{pmatrix} \begin{pmatrix} n-j-1 \\ n-j-k \end{pmatrix}. \]
In particular,
\[ A_n \cong NA_n \oplus \left( \bigoplus\limits_{i=1}^{n-1} NA_i^{\oplus |\boxcat_{\degen{}{}, \conn{}{}-}([1]^n, [1]^i)|} \right). \]

By construction, the image of the right summand inclusion $\bigoplus\limits_{i=1}^{n-1} NA_i^{\oplus |\boxcat_{\degen{}{}, \conn{}{}-}([1]^n, [1]^i)|} \ito A_n$ is contained in subgroup $\D A_n$.
It suffices to show the reverse inclusion.
Given $x \in A_{n-1}$, we may write
\[ x = y + \sum\limits_{i=1}^{n-2} \sum\limits_{\alpha \in \boxcat_{\degen{}{}, \conn{}{}-}([1]^{n-1}, [1]^i)} y_{\alpha} \alpha,  \]
where $y \in NA_{n-1}$ and $y_\alpha \in NA_i$ for each $\alpha \in \boxcat_{\degen{}{}, \conn{}{}-}([1]^{n-1}, [1]^i)$.
For $j \in \{ 1, \dots, n \}$, we have that $x\degen{}{j} = \sum y_\alpha \alpha \degen{}{j}$, which shows that $x \sigma_j$ indeed lies in the image of the inclusion $\bigoplus\limits_{i=1}^{n-1} NA_i^{\oplus |\boxcat_{\degen{}{}, \conn{}{}-}([1]^n, [1]^i)|} \ito A_n$, since each $y_\alpha$ is an element of some $NA_i$.
A similar argument shows that, for $k \in \{ 1, \dots, n-1 \}$, the $n$-cube $x\conn{}{k} \in A_n$ also lies in the image of the inclusion. 
\end{proof}
\begin{remark} \label{rmk:na_plus_da}
    For cubical sets with positive connections, an altered definition of $NA_\bullet$ which allows the $\face{}{n,1}$-face to be non-zero (and not the $\face{}{n,0}$-face) yields the analogous result
    \[ A_n \cong NA_n \oplus (\D[\degen{}{}]A_n + \D[\conn{}{}+] A_n). \]
\end{remark}

Using \cref{thm:na_oplus_da}, we may prove \cref{thm:pi_n_h_n}.
\begin{proof}[Proof of \cref{thm:pi_n_h_n}]
    Applying the second isomorphism theorem to \cref{thm:na_oplus_da} gives a chain complex isomorphism
    \[ A_\bullet / \D A_\bullet \cong NA_\bullet. \]
    This statement then follows from \cref{thm:pi_n_h_n_na}.
\end{proof}
\begin{remark} \label{rmk:counter-example}
    The proof of \cref{thm:pi_n_h_n} relies on quotienting by one connection when defining the Moore complex (as mentioned in \cref{rmk:def}).
    In the case of quotienting by positive connections, an analogous proof applies (\cref{rmk:na_plus_da}).
    When quotienting by both connections, one may want that
    \[ A_n \cong NA_n \oplus (\D[\degen{}{}]A_n + \D[\conn{}{}-] A_n + \D[\conn{}{}+] A_n). \]
    However, this is false since the intersection $NA \cap (\D[\degen{}{}] A_n + \D[\conn{}{}-] A_n + \D[\conn{}{}+] A_n)$ may be non-empty.
    This is the case for $A = F\cube{1} / F\{ 0 \}$, since $\sigma^2_1 + \sigma^2_2 - \conn{2}{1,1}-\conn{2}{1,0} \in (F\cube{1}/F\{ 0 \})_2$ is a square whose faces are all zero.
    Note that the Hurewicz theorem (\cref{thm:hurewicz1,thm:hurewicz}) still holds in this case since the associated homologies are isomorphic (\cref{rmk:iso}).
\end{remark}

Having established these facts, we can now move towards defining the Hurewicz homomorphism and proving the Hurewicz theorem.
The key ingredient in our approach is a natural transformation $\alpha \colon UF\loopsp \to \loopsp UF$ which we now define.

For a pointed Kan complex $(X,x)$, define a map $\alpha_X \from UF\loopsp X \to \loopsp UF X$ by regarding a linear combination $\sum a_i f_i$ of $n$-cubes $f_i \in (\loopsp X)_n$ as a linear combination of $(n+1)$-cubes of $X$ whose $\face{}{1,0}$- and $\face{}{1,1}$-face are degenerate at $x = 0$.
The map $\alpha_X$ satisfies the following properties.

\begin{proposition} \label{thm:alpha_eta_triangle}
    Let $X$ be a pointed Kan complex. 
    \begin{enumerate}
        \item The map $\alpha_X \from UF\loopsp X \to \loopsp UF X$ is natural in $X$.
        \item The map $\alpha_X \from UF\loopsp X \to \loopsp UFX$ is a monomorphism.
        \item The triangle
        \[ \begin{tikzcd}
            {} & UF\loopsp X \ar[dr, "\alpha_X"] & \\
            \loopsp X \ar[ur, "\eta_{\loopsp X}"] \ar[rr, "\loopsp \eta"] && \loopsp UFX
        \end{tikzcd} \]
        commutes.
    \end{enumerate}
\end{proposition}
\begin{proof}
  Items 1 and 3 are clear by definition.
  For item 2, observe that the image of the basis cubes $(\loopsp X)_n \subseteq (UF \loopsp X)_n$ under $(\alpha_X)_n \from (UF\loopsp X)_n \to (\loopsp UFX)_n$ is a set of linearly-indepenedent $(n+1)$-cubes of $UFX$. 
\end{proof}

\begin{proposition} \label{thm:alpha_iso}
    Let $X$ be a Kan complex.
    If, for some $n \geq 0$ and $x \in X_0$, we have that $X$ has a single $k$-cube $X_k = \{ x\ndegen[k] \}$ for all $0 \leq k \leq n$ then the map 
    \[ (\alpha_X)_k \from (UF \loopsp X)_k \to (\loopsp UFX)_k  \]
    is a bijection for $0 \leq k \leq n+1$.
\end{proposition}
\begin{proof}
    The map $\alpha_X$ is monic by \cref{thm:alpha_eta_triangle}, thus it suffices to show surjectivity.

    For $k \leq n-1$, this is immediate since $(\loopsp UFX)_k$ is a singleton set.
    For $k = n$, this holds since $(\loopsp X)_n$ contains all $(n+1)$-cubes of $X$ by assumption.

    For $k = n+1$, fix $\sum\limits_{i=1}^m a_i x_i \in (\loopsp UFX)_{n+1}$.
    By definition of the loop space, $\sum\limits_{i=1}^m a_i x_i$ is an $(n+2)$-cube of $FX$ whose $\face{}{1,0}$- and $\face{}{1,1}$-faces are 0.
    We apply \cref{thm:na_oplus_da} and write
    \[ \sum\limits_{i=1}^m a_i x_i = \sum\limits_{i=1}^{m_1} b_i y_i + \sum\limits_{i=1}^{m_2} c_i z_i \degen{}{i} + \sum\limits_{i=1}^{m_3} d_i w_i \conn{}{i,0}, \]
    where $\sum\limits_{i=1}^{m_1} b_i y_i \in N(FX)_{n+2}$ and $\sum\limits_{i=1}^{m_2} c_i z_i \degen{}{i} + \sum\limits_{i=1}^{m_3} d_i w_i \conn{}{i,0} \in \D (FX)_{n+2}$.
    The $\face{}{1,0}$- and $\face{}{1,1}$-faces of $\sum\limits_{i=1}^{m_1} b_i y_i$ are 0 by assumption.
    Thus,
    \[ 0 = \sum\limits_{i=1}^{m_2} c_i z_i \degen{}{i} \face{}{1,0} + \sum\limits_{i=1}^{m_3} d_i w_i \conn{}{i,0} \face{}{1,0} = \sum\limits_{i=1}^{m_2} c_i z_i \degen{}{i} \face{}{1,1} + \sum\limits_{i=1}^{m_3} d_i w_i \conn{}{i,0} \face{}{1,1}. \]
    For $i \geq 2$, cubical identities give that 
    \[ z_i \degen{}{i}\face{}{1,0} = z_i \face{}{1,0} \degen{}{i-1} = 0 \degen{}{i-1} = 0. \] 
    A similar calculation gives that $z_i \degen{}{i} \face{}{1,1}, w_i \conn{}{i,0}\face{}{1,0}, w_i \conn{}{i,0} \face{}{1,1} = 0$ for $i \geq 2$.
    Thus, we have
    \[ b_1 y_1 \degen{}{1} \face{}{1,0} + c_1 z_1 \conn{}{1,0} \face{}{1,0} = 0 = b_1 y_1 \degen{}{1} \face{}{1,1} + c_1 z_1 \conn{}{1,0} \face{}{1,1}. \]
    Applying cubical identities to the equality on the right gives $z_1 = 0$, thus $y_1 = 0$.
    That is, the $\face{}{1,0}$- and $\face{}{1,1}$-faces of $\sum\limits_{i=1}^{m_2} c_i z_i \degen{}{i} + \sum\limits_{i=1}^{m_3} d_i w_i \conn{}{i,0}$ are 0.
    Therefore, $\sum\limits_{i=1}^m a_i x_i \in (UF\loopsp X)_{n+1}$.
\end{proof}
\begin{corollary} \label{thm:alpha_conn}
    Let $X$ be a pointed Kan complex.
    For $n \geq 1$, if $X$ is $n$-connected then
    \[ \alpha_X \from UF \loopsp X \to \loopsp UFX \]
    is $n$-connected. 
\end{corollary}
\begin{proof}
    When $X$ has a single $k$-cube for all $0 \leq k \leq n$, this follows by \cref{thm:alpha_iso}.
    For an arbitrary $X$, consider the maximal subcomplex $Y \subseteq X$ which has a single $k$-cube for all $0 \leq k \leq n-1$.
    That is,
    \[ Y_k := \begin{cases}
        \{ x\ndegen[k] \} & k < n \\
        \{ y \in X_k \mid y\face{k}{i_1}\dots\face{n}{i_{k-n}} \in X_{n-1} \text{ for all } \face{k}{i_1}\dots\face{n}{i_{k-n}} \from \cube{n} \to \cube{k} \} & k \geq n.
    \end{cases} \]
    This subcomplex is homotopy equivalent to $X$ by \cite[Thm.~4.10]{carranza-kapulkin:homotopy-cubical} and has a single $k$-cube for all $0 \leq k \leq n$.
\end{proof}

\begin{definition} \label{def:Hurewicz-map}
For a pointed Kan complex $X$, the \emph{Hurewicz homomorphism} is the composite
\[ \pi_n X \xrightarrow{\pi_n \eta_X} \pi_n UFX \xrightarrow{\cong} \pH_n X  \]
of $\pi_n \from \cSet_* \to \Grp$ applied to the unit map $\eta_X \from X \to UFX$ with the isomorphism in \cref{thm:pi_n_h_n}.
\end{definition}

We now prove the Hurewicz theorem for the case of ($0$-)connected cubical Kan complexes.
\begin{theorem} \label{thm:hurewicz1}
    For a ($0$-)connected pointed Kan complex $(X,x)$, the Hurewicz homomorphism
    \[ \pi_1(X, x) \to \pH_1(X, x) \]
    is surjective.
    Moreover, its kernel is the commutator $[\pi_1(X, x), \pi_1(X, x)]$.
\end{theorem}
\begin{proof}
    It suffices to show the map $\pi_0 \loopsp \eta_X \from \pi_0 \loopsp X \to \pi_0 \loopsp UF X$ is surjective and that the pre-image of $[0]$ is $[\pi_1 X, \pi_1 X]$.

    \cref{thm:alpha_eta_triangle} gives that the diagram
    \[ \begin{tikzcd}
        {} & UF \loopsp X \ar[rd, "\alpha_X"] \\
        \loopsp X \ar[rr, "\loopsp \eta_X"] \ar[ur, "\eta_{\loopsp X}"] && \loopsp UF X
    \end{tikzcd} \]
    commutes.
    Applying $\pi_0$ to this diagram, we obtain:
    \[ \begin{tikzcd}
        {} & \pi_0 UF \loopsp X \ar[rd, "\pi_0 \alpha_X"] \\
        \pi_0 \loopsp X \ar[rr, "\pi_0 \loopsp \eta_X"] \ar[ur, "\pi_0 \eta_{\loopsp X}"] && \pi_0 \loopsp UF X
    \end{tikzcd} \]
    \cref{thm:alpha_conn} gives that the right map is a bijection.
    From this, it suffices to show the left map $\pi_0 \eta_{\loopsp X}$ is surjective and the preimage of $x\degen{1}{1}$ is the commutator $[\pi_1 X, \pi_1 X]$.

    Each basis cube is in the image of $\eta_{\loopsp X} \from \loopsp X \to UF \loopsp X$.
    Surjectivity then follows from \cref{thm:pi_n_a_n_mul}.

    To determine the pre-image of $x\degen{1}{1}$, we first note that $[\pi_1(X, x), \pi_1(X, x)]$ is contained in this pre-image by \cref{thm:pi_n_a_n_comm}.
    Showing the reverse inclusion, if $[a] \in \pi_0 \loopsp X$ is such that $[a] = [x\degen{}{1}]$ in $\pi_0 UF \loopsp X$, then there is an element $\sum\limits_{i=1}^n c_i u_i \in F X_2$ whose $\face{}{2,0}$-face is $a$ and whose other faces are $x\degen{}{1}$.
    For $c_i \neq -1, 1$, we write $c_i u_i$ using multiple sums so, without loss of generality, we may assume $c_i = \pm 1$ for all $i = 1, \dots, n$.
    As $\sum\limits_{i=1}^n c_i u_i \face{}{2,1} = x\degen{}{1}$, uniqueness of linear combinations gives that, for each term $u_j$, there is a term $u_k$ such that $u_j \face{}{2,1} = u_k \face{}{2,1}$ and $c_j = -c_k$. 
    As $\sum\limits_{i=1}^n c_i u_i \face{}{2,0} = a$, we have that for each term $u_j$, there is a term $u_k$ such that $u_j \face{}{2,0} = u_k \face{}{2,0}$ and $c_j = -c_k$ except for some $u_i$ such that $u_i \face{}{2,0} = a$ and $c_i = 1$.
    It follows that $\prod\limits_{i=1, \dots, n} [u_i \face{}{2,0}]^{c_i} = [a]$ and $\prod\limits_{i=1, \dots, n} [u_i \face{}{2,1}]^{c_i} = [x\degen{1}{1}]$ in the abelianization $\pi_1X / [\pi_1 X, \pi_1 X]$.
    As the $\face{}{1,0}$- and $\face{}{1,1}$-faces of each $u_i$ are $x\degen{}{1}$, we may horizontally compose the squares $u_1, \dots, u_n$ to obtain a square $u \in X_2$ such that
    \[ \begin{array}{l l}
        {[u\face{}{1,0}] = [x\degen{1}{1}]} & {[u\face{}{1,1}] = [x\degen{1}{1}]} \\
        {[u\face{}{2,0}] = \prod\limits_{\substack{i = 1, \dots, n}} [u_i\face{}{2,0}]^{c_i}} & {[u\face{}{2,1}] = \prod\limits_{\substack{i = 1, \dots, n}} [u_i\face{}{2,1}]^{c_i}}.
    \end{array} \]
    This square witnesses the equality $\prod\limits_{i=1, \dots, n} [u_i \face{}{2,0}]^{c_i} = \prod\limits_{i=1, \dots, n} [u_i \face{}{2,1}]^{c_i}$.
    This implies that $[a] = [x\degen{1}{1}]$ in the abelianization, thus $[a]$ lies in the commutator subgroup $[\pi_1 X, \pi_1 X]$.
\end{proof}
The result for $(n-1)$-connected Kan complexes follows.
\begin{theorem} \label{thm:hurewicz}
    Let $(X,x)$ be a pointed Kan complex.
    For $n \geq 2$, if $X$ is $(n-1)$-connected then the Hurewicz homomorphism
    \[ \pi_{n} (X, x) \to \pH_{n} (X,x) \]
    is an isomorphism.
\end{theorem}
\begin{proof}
    It suffices to show $\pi_0 \loopsp[n] \eta_X \from \pi_0 \loopsp[n] X \to \pi_0 \loopsp[n] UFX$ is a bijection.

    The diagram
    \[ \begin{tikzcd}
        {} & {\loopsp UF \loopsp[n-1] X} \ar[d, "{\loopsp \alpha_{\loopsp[n-2] X}}"] \\
        {} & {\loopsp[2] UF \loopsp[n-2] X} \ar[d, "{\loopsp[2] \alpha_{\loopsp[n-3] X}}"] \\
        {} & \vdots \ar[d, "{\loopsp[n-1] \alpha_X}"] \\
        {\loopsp[n] X} \ar[uuur, "\loopsp \eta_{\loopsp[n-1] X}", bend left] \ar[r, "{\loopsp[n] \eta_X}"] & {\loopsp[n] UFX}
    \end{tikzcd} \]
    commutes by repeated application of \cref{thm:alpha_eta_triangle}.
    Applying $\pi_0$, each map on the right side of the diagram
    \[ \begin{tikzcd}
        {} & {\pi_0 \loopsp UF\loopsp[n-1] X} \ar[d, "{\pi_0 \loopsp \alpha_{\loopsp[n-2] X}}", "\cong"'] \\
        {} & {\pi_0 \loopsp[2]UF \loopsp[n-2] X} \ar[d, "{\pi_0 \loopsp[2] \alpha_{\loopsp[n-3] X}}", "\cong"'] \\
        {} & \vdots \ar[d, "{\pi_0 \loopsp[n-1] \alpha_X}", "\cong"'] \\
        {\pi_0 \loopsp[n] X} \ar[uuur, "\pi_0 \loopsp \eta_{\loopsp[n-1] X}", bend left] \ar[r, "{\pi_0 \loopsp[n] \eta_X}"] & {\pi_0 \loopsp[n] UFX}
    \end{tikzcd} \]
    is a bijection by \cref{thm:alpha_conn}.
    The left map $\pi_0 \loopsp \eta_{\loopsp[n-1] X}$ is the Hurewicz homomorphism applied to the 0-connected Kan complex $\loopsp[n-1] X$.
    By \cref{thm:hurewicz1}, this map is surjective and its kernel is the commutator $[\pi_1 \loopsp[n-1] X, \pi_1 \loopsp[n-1]X]$.
    This commutator is trivial as $\pi_1 \loopsp[n-1] X \cong \pi_n (X, x)$ is abelian by \cite[Cor.~4.3]{carranza-kapulkin:homotopy-cubical}.
    Thus, the bottom map $\pi_0 \loopsp[n] \eta_X$ is a bijection.
\end{proof}

This completes the proof of the Hurewicz theorem.
We conclude by discussing the difficulties one might encounter trying to adapt this approach to the simplicial setting.

\begin{remark}
	The key feature of cubical sets that makes this proof possible is the fact that the $n$-cubes of $\loopsp X$ are just certain $(n+1)$-cubes of $X$ itself.
	By contrast, one typically defines the loop space of a \emph{simplicial} pointed Kan complex $X$ by
	\[ \loopsp (X, x)_n = \{ s \colon \Delta^1 \times \Delta^n \to X \ | \ \restr{s}{\{\varepsilon\} \times \Delta^n} = x \text{ for } \varepsilon = 0, 1\}. \]
	In particular, a single $n$-simplex consists of a collection of $2n$ simplicies of $X$ of dimension $(n+1)$, subject to the appropriate compatibility conditions.
	As a result, simplicial proofs of an analogous fact need to either involve more sophisticated combinatorics \cite[Thm.~13.6]{may:simplicial-objects} or rely on other less elementary tools, e.g., Serre spectral sequence (cf.~\cite[Thm.~III.3.7]{goerss-jardine}).
\end{remark}